\documentclass[reqno, 10pt, a4paper]{amsart}%

\usepackage{graphicx}
\usepackage{multicol,multirow}
\usepackage{amsmath,amssymb,amsfonts}
\usepackage{mathrsfs}
\usepackage{rotating}
\usepackage{appendix}
\usepackage[numbers]{natbib}

\numberwithin{equation}{section}

\usepackage[utf8]{inputenc}
\usepackage[OT2,T1]{fontenc}
\usepackage[english]{babel}

\usepackage{amsmath}
\usepackage{amssymb}
\usepackage{amsthm}

\usepackage{fullpage}
\selectfont

\usepackage[mathcal]{euscript}

\usepackage[dvipsnames]{xcolor}
\usepackage{graphicx}
\usepackage{pgfplots}
\pgfplotsset{compat=1.17}
\usepgfplotslibrary{groupplots}

\usepackage{booktabs}

\usepackage[backref=page, hyperindex=true, colorlinks=true,
  breaklinks=true, urlcolor=blue, linkcolor=blue, citecolor=blue,
  bookmarks=true, bookmarksopen=true]{hyperref}

\usepackage{colonequals}

\theoremstyle{plain}
\newtheorem{theorem}{Theorem}[section]
\newtheorem{proposition}[theorem]{Proposition}
\newtheorem{lemma}[theorem]{Lemma}
\newtheorem{corollary}[theorem]{Corollary}
\newtheorem{definition}[theorem]{Definition}
\newtheorem{conjecture}[theorem]{Conjecture}

\theoremstyle{remark}
\newtheorem{remark}[theorem]{Remark}
\newtheorem{example}[theorem]{Example}

\numberwithin{equation}{section}

\newcommand\eps{\varepsilon}

\DeclareMathOperator{\Aut}{Aut}
\DeclareMathOperator{\Fr}{Fr}
\DeclareMathOperator{\Fd}{\mathfrak{f}}
\DeclareMathOperator{\GSp}{GSp}
\DeclareMathOperator{\Hd}{\mathfrak{h}}
\DeclareMathOperator{\PGL}{PGL}
\DeclareMathOperator{\RedAut}{RedAut}
\DeclareMathOperator{\sgn}{sgn}
\DeclareMathOperator{\Tr}{Tr}
\DeclareMathOperator{\USp}{USp}
\DeclareMathOperator{\Twist}{Twist}

\newcommand\cprime{\/{\mathsurround=0pt$'$}}

\newcommand\Fbar{\overline{\F}}
\newcommand\Xbar{\overline{\X}}
\newcommand\Qbar{\overline{\Q}}

\newcommand\CC{\mathbb{C}}
\newcommand\F{\mathbb{F}}
\newcommand\NN{\mathbb{N}}
\newcommand\Q{\mathbb{Q}}
\newcommand\R{\mathbb{R}}
\newcommand\VV{\mathbb{V}}
\newcommand\Z{\mathbb{Z}}

\newcommand\bS{\mathbf{S}}
\newcommand\bs{\mathbf{s}}

\newcommand\calH{\mathcal{H}}
\newcommand\I{\mathcal{I}}
\newcommand\M{\mathcal{M}}
\newcommand\N{\mathcal{N}}
\newcommand\X{\mathcal{X}}
\newcommand\calV{\mathcal{V}}
\newcommand\Y{\mathcal{Y}}

\newcommand\fraka{\mathfrak{a}}
\newcommand\bb{\mathfrak{b}}
\newcommand\cc{\mathfrak{c}}

\newcommand{\Mnh}{\mathcal{M}^{\nhyp}}
\newcommand{\Nq}{\mathcal{N}}

\newcommand{\even}{\textup{even}}
\newcommand{\odd}{\textup{odd}}
\newcommand{\hyp}{\textup{hyp}}
\newcommand{\nhyp}{\textup{nhyp}}

\hypersetup{
  pdfauthor   = {Bergstr\"om, Howe, Lorenzo Garc\'ia, Ritzenthaler},
  pdftitle    = {Refinements of Katz-Sarnak theory for the number of points on curves over finite fields},
  pdfsubject  = {},
  pdfkeywords = {},
}

\begin{document}
\title[Refinements of Katz--Sarnak theory for the number of points on curves over finite fields]{Refinements of Katz--Sarnak theory for the number of points on curves over finite fields}

\date{14 December 2023}

\author[Bergstr\"om]{Jonas Bergstr\"om}
\address{%
Jonas Bergstr\"om, 
Matematiska institutionen, Stockholms Universitet, SE-106 91, Stockholm, Sweden
}
\email{jonasb@math.su.se}

\author[Howe]{Everett W. Howe}
\address{%
Everett W. Howe,
Independent mathematician,
San Diego, CA 92104 U.S.A.
}
\email{however@alumni.caltech.edu}
\urladdr{\href{https://ewhowe.com}{https://ewhowe.com}}

\author[Lorenzo]{Elisa Lorenzo García}
\address{%
	Elisa Lorenzo Garc\'ia,
  Universit\'e de Neuch\^atel, rue Emile-Argand 11, 2000, Neuch\^atel,
  Switzerland. 
}

\address{%
	Elisa Lorenzo Garc\'ia,
   Univ Rennes, CNRS, IRMAR - UMR 6625, F-35000
 Rennes, %
  France. %
}
\email{elisa.lorenzo@unine.ch, elisa.lorenzogarcia@univ-rennes1.fr}

\author[Ritzenthaler]{Christophe Ritzenthaler}
\address{%
	Christophe Ritzenthaler,
  Univ Rennes, CNRS, IRMAR - UMR 6625, F-35000
 Rennes, %
  France. %
  }
  
\address{%
	Christophe Ritzenthaler,
  Université Côte d'Azur, CNRS, LJAD UMR 7351,
  Nice,
  France
}
\email{christophe.ritzenthaler@univ-rennes1.fr}




\begin{abstract}
This paper goes beyond Katz--Sarnak theory on the distribution of curves over finite fields according to their number of rational points, theoretically, experimentally and conjecturally. In particular, we give a formula for the limits of the moments measuring the asymmetry of this distribution for (non-hyperelliptic) curves of genus $g\geq 3$. The experiments point to a stronger notion of convergence than the one provided by the Katz--Sarnak framework for all curves of genus $\geq 3$. However, for elliptic curves and for hyperelliptic curves of every genus we prove that this stronger convergence cannot occur.   
 \end{abstract}

\maketitle

\section{Introduction}
\label{sec:intro}

Katz--Sarnak theory \cite{katz-sarnak} gives a striking unified framework to understand the distribution of the traces of Frobenius for a family of curves\footnote{ 
   Throughout this paper, the word `curve' will always mean a
   projective, absolutely irreducible, smooth variety of dimension~$1$.}
of genus $g$ over a finite field $\F_q$ when $q$ goes to infinity. It has been used in many specific cases: see
\cite{AS10}, \cite{BCDGL}, \cite{CDSS}, \cite{FKRS}, \cite{HKLLM}, \cite{KS09}, \cite{vladut},
among others. 
Although powerful, this theory can neither in general predict the number of curves 
over a given finite field with a given trace, nor distinguish between the family of all
curves and the family of hyperelliptic curves when $g\geq 3$. This paper can be seen as
an attempt to go beyond Katz--Sarnak results, theoretically, experimentally and 
conjecturally. We hope that this blend will excite the curiosity of the community.

We begin by resuming our study of sums of powers of traces initiated in  \cite{BHLR22}.
If $C/\F_q$ is a curve of genus $g$, we denote by $[C]$ the set of representatives of its twists and define 
\[
s_n(C) = \sum_{C' \in [C]} \frac{(q+1-\#C'(\F_q))^n}{\# \Aut_{\F_q}(C')}.
\]
As shown in \cite[Prop.~3.1]{BHLR22}, the $s_n(C)$ are integers and 
we denote by $S_{n}(q,\X)$ the sum of the $s_n(C)$ when $C$ runs over a set of representatives for the $\Fbar_q$-isomorphism classes of curves $C$ over $\F_q$ in $\X$, where $\X$ can be, for example,
\begin{itemize}
\item the moduli space $\M_{1,1}$ of elliptic curves, 
\item the moduli space $\M_g$ of curves of genus $g >1$, 
\item the moduli space $\calH_g$ of hyperelliptic curves of genus $g>1$, or 
\item the moduli space $\Mnh_g$ of non-hyperelliptic curves of genus $g>2$.
\end{itemize}
In Remark~\ref{remark-elliptic} we will briefly recall that $S_n(q,\M_{1,1})$ can be determined for all $q$ and $n$ in terms of traces of Hecke operators on spaces of elliptic modular cusp forms.
For every $q$ and $n$ we can also find expressions for $S_n(q,\M_2)=S_n(q,\calH_2)$ in terms of traces of Hecke operators acting on spaces of Siegel modular cusp forms of genus $2$ (and genus $1$) starting from \cite[Thm.~2.1]{petersen}; see \cite[\S4.5]{BF22} for a few more details. 
For every $g \geq 3$, there are known explicit formulae for $S_n(q,\X)$ only for the first values of $n$; see for instance \cite[Thm.~3.4]{BHLR22} for $\calH_g$ (note that the odd $n$ values are equal to $0$ in this case) and \cite{berg} for $\Mnh_3$. However, it is possible to give an interpretation for
\[
\fraka_n(\X)\colonequals
\lim_{q \to \infty} \frac{S_n(q,\X)}{q^{\dim \X+n/2}}
\]
with $\X=\M_g$, $\calH_g$ or $\Mnh_g$ for every $g \geq 2$ and even $n \geq 2$ in terms of representation theory of the compact symplectic group $\USp_{2g}$. This is achieved in \cite[Thm.~3.8]{BHLR22} 
using the ideas of Katz and Sarnak. 

Our first contributions are gathered in Theorem~\ref{theorem:H0H1}. Using the results of Johnson \cite{johnson} and Hain \cite{hain}, together with results of Petersen \cite{petersen,petersen3}  about the first cohomology group of symplectic local systems on $\M_g$, we can prove that for even values of $n>0$ we have
\begin{equation} \label{eq:conv-an}
\fraka_n(\M_g) - \frac{S_n(q,\M_g)}{q^{\dim \M_g+n/2}} = O(q^{-1})
\end{equation}
when $g\geq 2$,
whereas Katz--Sarnak would only give $O(q^{-1/2})$. Since $\fraka_n(\M_g)=0$ for odd values of $n$, this suggests replacing the exponent in the power of $q$ in the denominator of the expression defining $\fraka_n(\M_g)$ with a smaller number. As far as we know this has not been considered previously. We therefore introduce for odd $n$
\[
\bb_n(\M_g)\colonequals
-\lim_{q \to \infty} \frac{S_n(q, \M_g)}{q^{3g-3+(n-1)/2}}.
\]
Theorem~\ref{theorem:H0H1} gives $\bb_n(\M_g)$  in terms of an explicit integral and in terms of the representation theory of $\USp_{2g}$. This second description  makes it easy to compute. The idea to use information about the cohomology of moduli space of curves to predict the number of curves over a given finite field with a given trace can also be found in \cite{AEKWZB}, but there $g$ goes to infinity. 
\\

The deep relations between the sum of traces and Katz--Sarnak theory becomes clearer once we switch to a probabilistic point of view. In Section~\ref{sec:cvmoments}, we introduce the classical probability measure $\mu_{q,g}$ on the interval $[-2g,2g]$ derived from the numbers of $\F_q$-isomorphism classes of curves of genus $g>1$ with given traces of Frobenius. From Katz--Sarnak, we then know that the sequence of measures $(\mu_{q,g})$ weakly converges to a continuous measure $\mu_g$ with an explicit density $\Fd_g$ (see \cite[Thm.~2.1]{billingsley} for equivalent definitions of weak convergence of measures). In this language, the numbers $\fraka_n(\M_g)$ can be understood as the $n$th moments of the measure $\mu_g$, and 
 we can refine Katz--Sarnak theory using a second continuous function $\Hd_g$ whose $n$th moments are the numbers $\bb_n(\M_g)$ (see Theorem~\ref{th:betterKS}).\\

In Section~\ref{sec:experiment}, we
investigate whether the Katz--Sarnak limiting distributions can be
used to approximate the number of curves over a given finite field
$\F_q$ of a given genus and with a given trace of Frobenius; 
one might hope that integrating that 
distribution over an interval of
length $1/\sqrt{q}$ around $t/\sqrt{q}$  would give a value close
to the number of genus-$g$ curves over $\F_q$ having trace~$t$.
We show that this does \emph{not} happen for elliptic curves or for
hyperelliptic curves of any genus. For elliptic curves,
Proposition~\ref{P:nolimit1} shows that the number of elliptic curves
with a given trace can be an arbitrarily large multiple of this
na\"{\i}ve Katz--Sarnak prediction (see also Figure~\ref{fig:genus1}). For hyperelliptic curves,
Proposition~\ref{P:nolimit} shows (roughly speaking)
that if the number of curves is
asymptotically bounded above and below by two multiples of the
na\"{\i}ve Katz--Sarnak prediction, then the ratio of these
two multiples is bounded below by a fixed number strictly
greater than~$1$ (see Figure~\ref{fig:genus2}).

On the other hand, numerical experiments suggest that the elliptic and hyperelliptic cases differ in the sense that it is easy to `correct' the distribution in the hyperelliptic cases to observe a good approximation by the density function $\Fd_g$ (see Figure~\ref{fig:genus2scaled}). Even stronger, computations for all non-hyperelliptic curves of genus $3$ (see Figure~\ref{fig:genus3}) make us dream that
the na\"{\i}ve Katz--Sarnak approximation \emph{does} directly give an accurate estimate for the number of curves with a given number of points. This leads us to claim the bold Conjecture~\ref{conj:ponctual-conv}. The heuristic idea behind this conjecture is that for each trace, one is averaging over many isogeny classes which somehow would allow this stronger convergence as long as there are no obvious arithmetic obstructions. Our attempts to use the better convergence rates of the moments in the case of $\M_g$ for $g \geq 3$ to prove this conjecture were unfortunately unsuccessful. {However, for $g=1$, we would like to point out the shortening of the intervals of convergence obtained in \cite{ma2023refinements}, which may give some hints for addressing the question.}\\

Finally, in Section~\ref{sec:asymmetry} we  
revisit the work of \cite{RRRSS} on the symmetry breaking for the trace distribution of (non-hyperelliptic) genus $3$ curves, by looking at the difference between the number of curves with trace $t$ and the number of curves with trace $-t$. In probabilistic terms, this asymmetry is given by a signed measure $\nu_{q,g}$. Although this signed measure weakly converges to $0$ when $q$ goes to infinity, by Corollary~\ref{cor:momentsnu} the moments of $\sqrt{q} \,\nu_{q,g}$ converge to $-2 \bb_n(\M_g)$ when $n$ is odd (and are trivially $0$ when $n$ is even).  In particular, this shows that by `zooming in' on the Katz--Sarnark distribution, one can spot a difference between the behaviour for hyperelliptic curves (for which the corresponding signed measures would all be $0$) and for non-hyperelliptic curves.

In the same spirit as Section~\ref{sec:experiment}, the experimental data for $g=3$ (see Figure~\ref{fig:comp}) and the convergence of moments lead us to conjecture that the sequence of signed measures $(\sqrt{q} \,\nu_{q,g})$ weakly converges to the continuous signed measure with density $-2 \Hd_g$ for all $g \geq 3$. Notice that in contrast to the case of positive bounded measures, the convergence of moments of signed measures on a compact interval does not directly imply weak convergence; see Example~\ref{ex:moments-signed}.

With such a conjecture in hand, one may then improve on the result of \cite{RRRSS} that 
heuristically approximated the limit density of $(\sqrt{q} \,\nu_{q,g})$ by the function $$x (1-x^2/3) \cdot \left(\frac{1}{\sqrt{2\pi}} e^{-x^2 / 2}\right).$$ Using the first values of $\bb_n(\M_3)$, we get the better approximation
\[
x \left(5/4- x^2/2+ x^4/60\right) \left(\frac{1}{\sqrt{2 \pi}} e^{-x^2/2}\right).
\]

\subsection*{Acknowledgement} We thank Dan Petersen for helpful conversations in connection with the Gross--Schoen cycle and Sophie Dabo for discussions on measure theory.


\section{Limits of sums of powers of traces} \label{sec:moments}
Fix a prime power $q$. 
Let us start by recalling some definitions and results from  \cite{BHLR22}.
\begin{definition}  \text{ }
Let $\X=\calH_g$, $\M_g$ or $\Mnh_g$ for any $g \geq 2$, or $\X=\M_{1,1}$. 
\begin{itemize}
\item[$\star$] Recall from Section~\textup{\ref{sec:intro}} that one defines 
\[
S_n(q,\X)=\sum_{[C] \in \X(\F_q)} \sum_{C' \in [C]} \frac{(q+1-\#C'(\F_q))^n}{\# \Aut_{\F_q}(C')}
\]
where $[C]$ is a point of $\X(\F_q)$ representing the $\Fbar_q$-isomorphism class of a curve $C/\F_q$, and the second sum spans the set of representatives of all twists $C'$ of $C$.
\item[$\star$] For every $n \geq 1$, let 
\[ 
\fraka_n(\X)\colonequals
\lim_{q \to \infty} \frac{S_n(q,\X)}{q^{\dim \X+n/2}}
\]
with $\X=\calH_g$ or $\M_g$ or $\Mnh_g$ for any $g \geq 2$, or with $\X=\M_{1,1}$. 
\end{itemize}
\end{definition}
 Define $w_k\colonequals\sum_{j=1}^g  2\cos k\theta_j$ and
\[
dm_g\colonequals\frac{1}{g!\,\pi^g}\prod_{i<j}(2\cos\theta_i-2\cos\theta_j)^2\prod_i 2\sin^2\theta_i\  d\theta_1\ldots d\theta_g,
\] and recall from \cite[Thm.~2.1]{BHLR22}
that for every $g \geq 2$ and $n \geq 1$, 
\[
\fraka_n(\X)= \int_{{(\theta_1,\ldots,\theta_g)}\in [0,\pi]^g } w_1^n  \, d m_g,
\]
with $\X=\calH_g$ or $\M_g$ or $\Mnh_g$. Notice that for a fixed value of $g$, $\fraka_n(\X)$ does not depend on $\X$, and 
$\fraka_n(\X)=0$ for odd $n$.

In order to go deeper in the limit distribution, we will also look at the `next term' of the limit of $\frac{S_n(q,\X)}{q^{\dim \X+n/2}}$ when $\X=\M_g$. 
\begin{definition}
For every $g \geq 2$ and $n \geq 1$, let 
\[ 
\bb_n(\M_g)\colonequals
-\lim_{q \to \infty} \sqrt{q} \left( \frac{S_n(q, \M_g)}{q^{3g-3+n/2}}-\fraka_n(\M_g)\right). 
\]    
\end{definition}

To state our results, we need to recall basic facts about the representations of $\USp_{2g}$ with coefficients in $\Q_{\ell}$, where $\ell$ is a prime distinct from the characteristic of $\F_q$. The irreducible representations $V_{\lambda}$ of $\USp_{2g}$ are indexed by the highest weight $\lambda=(\lambda_1,\ldots,\lambda_g)$ with $\lambda_1 \geq \ldots \geq \lambda_g \geq 0$. The corresponding characters $\chi_{\lambda}$ are the symplectic Schur polynomials $\bs_{\langle\lambda\rangle}(x_1,\ldots,x_g) \in \Z[x_1,\ldots,x_g,x_1^{-1},\ldots,x_g^{-1}]$ in the sense that if $A \in \USp_{2g}$ has eigenvalues $\alpha_1,\ldots,\alpha_g,\alpha_1^{-1},\ldots,\alpha_g^{-1}$ then $\chi_{\lambda}(A)=\bs_{\langle\lambda\rangle}(\alpha_1,\ldots,\alpha_g)$; see \cite[Prop.~24.22 and (A.45)]{Fulton-Harris}. In the notation we will suppress the $\lambda_j$ that are $0$. Put $\lvert\lambda\rvert=\lambda_1+\ldots+\lambda_g$ and note that $V_{\lambda}^{\vee} \cong V_{\lambda}$. 

\begin{theorem} \label{theorem:H0H1} Let $V=V_{(1)}$ denote the standard representation.
\begin{enumerate}
\item \label{item:1}
Let $\X=\calH_g$, $\M_g$, $\Mnh_g$ for any $g \geq 2$, or $\M_{1,1}$. For every $n \geq 1$, 
$\fraka_n(\X)$ is equal to  the number of times the trivial representation appears in the $\USp_{2g}$-representation $V^{\otimes n}$. \footnote{This is  precisely \cite[Thm.~3.8]{BHLR22}, but we will give a different proof.}  
\item For every $g \geq 3$ and $n \geq 1$, $\bb_n(\M_g)$ is equal to the number of times the representation $V_{(1,1,1)}$ appears in the $\USp_{2g}$-representation $V^{\otimes n}$. In particular $\bb_n(\M_g)=0$ for $n$ even. \label{item:2}
\item For every $n \geq 1$, $\bb_n(\M_2)=0$. \label{item:3}
\item \label{item:4} For every $g \geq 2$ and $n \geq 1$, 
\[
\fraka_n(\M_g)-\frac{\bb_n(\M_g)}{\sqrt{q}}=  \frac{S_n(q,\M_g)}{q^{3g-3+n/2}}
  +O(q^{-1}).
 \]
\item \label{item:5} For every $g \geq 3$ and $n \geq 1$ we have 
\begin{equation} \label{eq:bn}
\bb_n(\M_g)= \int_{{(\theta_1,\ldots,\theta_g)}\in [0,\pi]^g } w_1^n \Bigl(\frac{1}{6}w_1^3-\frac{1}{2}w_1w_2+\frac{1}{3}w_3-w_1 \Bigr) \, d m_g.
\end{equation}
\end{enumerate}
\end{theorem}

\begin{proof} Poincar\'e duality gives a symplectic pairing on the first $\ell$-adic \'etale cohomology group of a curve. We will be interested in the action of Frobenius on these cohomology groups and since we need to take the size of the eigenvalues of Frobenius into account we will consider representations of $\GSp_{2g}$. Let $\Q_{\ell}(-1)$ denote the \emph{multiplier representation} or \emph{similitude character}; if we identify $\GSp_{2g}$ as the group of automorphisms of a $2g$-dimensional vector space that preserve a symplectic form $s$ up to scaling, then $\Q_{\ell}(-1)$ is the representation $\eta$ that sends an element of $\GSp_{2g}(\Q_\ell)$ to the factor by which it scales~$s$. Let $\Q_{\ell}(1)$ be the inverse (or dual) of $\Q_{\ell}(-1)$, and for an integer $j$ put $\Q_{\ell}(j)=\Q_{\ell}(\sgn j)^{\otimes \lvert j\rvert}$. For a representation $U$ put $U(j)\colonequals U \otimes \Q_{\ell}(j)$. With the standard representation $W$ of $\GSp_{2g}$ we can get irreducible representations $W_{\lambda}$, for $\lambda=(\lambda_1,\ldots,\lambda_g)$ with $\lambda_1 \geq \ldots \geq \lambda_g \geq 0$, using the same construction as for $\USp_{2g}$; see \cite[(17.9)]{Fulton-Harris}. If we homogenize the polynomial $s_{\langle\lambda\rangle}(x_1,\ldots,x_g,t)$ to degree $\lvert\lambda\rvert$ using a variable $t$ of weight $2$ and with $x_i$ of weight $1$ for $i=1,\ldots,g$, then for $A \in \GSp_{2g}$ with $\eta(A)=s$ and eigenvalues $\alpha_1,\ldots,\alpha_g,s\alpha_1^{-1},\ldots,s\alpha_g^{-1}$ we have $\chi_{\lambda}(A)=s_{\langle\lambda\rangle}(\alpha_1,\ldots,\alpha_g,s)$. Now, for every $n$, there are integers $c_{\lambda,n} \geq 0$ such that 
\begin{equation} \label{eq:tensor}
W^{\otimes n} \cong \bigoplus_{\lvert\lambda\rvert \leq n} W_{\lambda}^{\oplus c_{\lambda,n}}\bigl((-n+\lvert\lambda\rvert)/2\bigr).
\end{equation}
Note that if $n \not \equiv \lvert\lambda\rvert \bmod 2$ then $c_{\lambda,n}=0$. Note also that \eqref{eq:tensor} holds with the same $c_{\lambda,n}$ when replacing $\GSp_{2g}$ with $\USp_{2g}$, i.e.~replacing $W$ by $V$ and ignoring the multiplier representation. Note also that $W_{\lambda}^{\vee} \cong W_{\lambda}(\lvert\lambda\rvert)$. 

Let $\X=\calH_g$, $\M_g$ or $\Mnh_g$ for any $g \geq 2$, or $\X=\M_{1,1}$. Let $\pi: \Y \to \X $ be the universal object and define the $\ell$-adic local system $\VV=R^1 \pi_{*} \Q_{\ell}$. To any irreducible representation of $\GSp_{2g}$ (the symplectic pairing coming as above from the first cohomology group of the curves) corresponding to $\lambda$ we can then use Schur functors to define a local system $\VV_{\lambda}$.  Let $H^j_c$ denote  compactly supported $\ell$-adic cohomology and $\Fr_q$ the geometric Frobenius acting on $\X \otimes \Fbar_q$. For general results on \'etale cohomology of stacks, see for instance \cite{sun}. 

For almost all primes $p$ we have $H^j_c(\X \otimes \CC, \VV_{\lambda}) \cong H^j_c(\X \otimes \Qbar_p, \VV_{\lambda}) \cong H^j_c(\X \otimes \Fbar_p, \VV_{\lambda})$. From this we get bounds on $\dim_{\Q_{\ell}} H^j_c(\X \otimes \Fbar_p, \VV_{\lambda})$ that are independent of $p$. This will tacitly be used below when we let $q$ go to infinity. 

Put $\Xbar=\X \otimes \Fbar_q$. The Lefschetz trace formula and \eqref{eq:tensor} then tell us that 
\begin{align*}
S_n(q,\X) 
&=\sum_{j=0}^{2\dim \X} (-1)^j\Tr(\Fr_q,H^j_c(\Xbar, \VV_1^{\otimes n}))\\
&=\sum_{\lambda} c_{\lambda,n}  \,  \sum_{j=0}^{2 \dim \X} (-1)^j \Tr(\Fr_q,H^j_c(\Xbar, \VV_{\lambda})) \, q^{(n-\lvert\lambda\rvert)/2}\,;
\end{align*}
compare \cite[\S8]{BFvdG}. 
Since $\VV_{\lambda}$ is pure of weight $\lambda$, it follows from Deligne's theory of weights \cite{deligne2,sun} that the trace of Frobenius on $H^j_c(\Xbar,\VV_{\lambda})$ is equal (after choosing an embedding of $\Qbar_{\ell}$ in~$\CC$) to a sum of complex numbers with absolute value at most $q^{(j+\lvert\lambda\rvert)/2}$. 
 
From this we see that only when $j=2\dim \X$ can we get a contribution to $\fraka_n(\X)$.
Since $\X$ is a smooth Deligne--Mumford stack, Poincar\'e duality shows that for every $i$ with $ 0 \leq i \leq 2\dim \X$, we have
\[
H_c^{2\dim \X-i}(\Xbar,\VV_{\lambda}) \cong H^i(\Xbar,\VV_{\lambda})^{\vee}(-\dim \X-\lvert\lambda\rvert).
\]
The zeroth cohomology group of a local system consists of the global invariants, and among the irreducible local systems, only the constant local system $\VV_{(0)} \cong \Q_{\ell}$ has such. Moreover, $H^0(\Xbar,\Q_{\ell})$  is one-dimensional, since $\X$ is irreducible. Finally, since the action of $\Fr_q$ on $H^0(\Xbar,\Q_{\ell})$ is trivial, we get by Poincar\'e duality that $\Fr_q$ acts on $H_c^{2\dim \X}(\Xbar,\Q_{\ell})$ by multiplication by $q^{\dim \X}$. It follows that $\fraka_n(\X)=c_{(0),n}$. This proves~\eqref{item:1}.

Assume now that $g \geq 3$. From the work of Johnson and Hain we know that $H^{1}(\M_g,\VV_{\lambda})$ is nonzero if and only if $\lambda=(1,1,1)$; see \cite{johnson}, \cite{hain} and \cite[Thm.~4.1 and Cor.~4.2]{kabanov}. In these references, it is the rational Betti cohomology group of $\M_g$ over the complex numbers that is considered. 
 Furthermore, $H^{1}(\M_g \otimes \Fbar_q,\VV_{(1,1,1)})$ is one-dimensional and generated by the Gross--Schoen cycle, which lives in the second Chow group; see see~\cite[Rem.~12.1 and Ex.~6.4]{petersenetal2}. Since this result also holds in $\ell$-adic cohomology, as noted in \cite[\S1.2]{petersenetal2}, the action of $\Fr_q$ on this cohomology group is by multiplication by~$q^2$

Recall that $\dim \M_g=3g-3$. By Poincaré duality we find that the action of $\Fr_q$ on $H_c^{6g-7}(\M_g \otimes \Fbar_q,\VV_{(1,1,1)})$ is by $q^{3g-3+3-2}$. We can now conclude the following. If $n$ is even then $c_{(1,1,1),n}=0$, and so every eigenvalue of Frobenius contributing to $q^{3g-3+n/2}c_{(0),n}-S_n(q,\M_g)$ has absolute value at most $q^{3g-4+n/2}$. If $n$ is odd then $c_{(0),n}=0$, and so there are no eigenvalues of Frobenius contributing to $S_n(q,\M_g)$ of absolute value  $q^{3g-3+n/2}$ and we can conclude by the above that $\bb_n(\M_g)=c_{(1,1,1),n}$. This proves~\eqref{item:2}. 

Because of the hyperelliptic involution, $H_c^{i}(\M_2,\VV_{\lambda})=0$ for all $\lambda$ such that $\lvert\lambda\rvert$ is odd. Moreover, $H^{1}(\M_2,\VV_{\lambda})$ is nonzero precisely when $\lambda=(2,2)$. It is then one-dimensional and $\Fr_q$ acts by multiplication by~$q^3$. This result is proven but not stated explicitly in \cite{petersen,petersen3}, as explained in \cite[Cor.~6.7]{watanabe}. By Poincaré duality, $\Fr_q$ acts on $H_c^{5}(\M_2,\VV_{2,2})$ by multiplication by~$q^{3+4-3}$. Hence, for all even $n$, every eigenvalue of Frobenius contributing to $q^{3+n/2}c_{(0),n}-S_n(q,\M_2)$ has absolute value at most $q^{3+(n-2)/2}$. This proves~\eqref{item:3}. 

Statement~\eqref{item:4} is only a reformulation of the properties of $\fraka_n(\M_g)$ and $\bb_n(\M_g)$ proven above.

Finally, for every $k \geq 1$, put $p_k(x_1,\ldots,x_g) \colonequals\sum_{i=1}^g (x_i^k+x_i^{-k})$. The polynomial $\bs_{\langle(1,1,1)\rangle}(x_1,\ldots,x_g)$ equals   
\[
\frac{1}{6}p_1^3-\frac{1}{2}p_1p_2+\frac{1}{3}p_3-p_1. 
\]
The irreducible representations of $\USp_{2g}$ are self-dual. As a consequence, if $U$ is a representation of $\USp_{2g}$ then the number of times the representation $V_{\lambda}$ appears in $U$ equals the number of times the trivial representation appears in $V_{\lambda} \otimes U$. If $A \in \USp_{2g}$ has eigenvalues $\alpha_1,\ldots,\alpha_g,\alpha_1^{-1},\ldots,\alpha_g^{-1}$, with $\alpha_j=e^{i\theta_j}$ for $j=1,\ldots,g$, then $p_k(\alpha_1,\ldots,\alpha_g)=w_k(\theta_1,\ldots,\theta_g)$. Statement~\eqref{item:5} now follows from \eqref{item:2}. 
\end{proof}

\begin{remark} \label{remark-elliptic} 
Why did we not define $\bb_n$ for $\M_{1,1}$? For every prime $p$ and $n > 0$ it follows from \cite{deligne3} (see also \cite{birch} and \cite[\S2]{BFvdG}) that 
\begin{align*}
\sum_{j=0}^{2} (-1)^j\Tr(\Fr_p,H^j_c(\M_{1,1} \otimes \Fbar_p, \VV_{(n)}))
&=-\Tr(\Fr_p,H^1_c(\M_{1,1} \otimes \Fbar_p, \VV_{(n)}))\\
&=-1-\Tr(T_p,\bS_{n+2}),  
\end{align*}
where $T_p$ is the $p$th Hecke operator acting on $\bS_{n+2}$, the (complex) vector space of elliptic modular cusp forms of level $1$ and weight $n+2$. Moreover, for every prime power $q$, the eigenvalues of $\Fr_q$ acting on $H^1_c(\M_{1,1} \otimes \Fbar_p, \VV_{(n)})$ will have absolute value $q^{(n+1)/2}$. It is in general not clear that the limit 
\begin{equation} \label{equation-limit}
-\lim_{q \to \infty} \sqrt{q} \left( \frac{S_n(q, \M_{1,1})}{q^{1+n/2}}-\fraka_n(\M_{1,1})\right),
    \end{equation}
which would be the way to define $\bb_n(\M_{1,1})$, always exists when $n$ is even. (For odd $n$, $S_n(q,\M_{1,1})=0$, hence the limit \eqref{equation-limit} will be $0$.)

For even $0 \leq n \leq 8$, the limit  \eqref{equation-limit} is also $0$ since there are no elliptic cusp forms level $1$ and weight less than or equal to $10$. We then have that $S_{10}(p,\M_{1,1})=42p^6-\Tr(T_p, \bS_{12})+O(p^5)$ and $S_{12}(p,\M_{1,1})=132p^7-11p \cdot \Tr(T_p, \bS_{12})+O(p^6)$. 
The so-called Frobenius angle, $0 \leq \varphi_p \leq \pi$, of the Hecke eigenform (the Ramanujan $\Delta$ function) in the one-dimensional space $\bS_{12}$ is defined by $a_p\colonequals\Tr(T_p,\bS_{12})=2p^{11/2}\cos \varphi_p$. The Sato--Tate conjecture for $\Delta$ (proven in \cite{BLGHT11}) then tells us that there are sequences of primes $p'_1,p'_2,\ldots$ and $p''_1,p''_2,\ldots$ such that the Frobenius angles of $a_{p'_1},a_{p'_2},\ldots$ (respectively $a_{p''_1},a_{p''_2},\ldots$) are all between $0$ and $\pi/3$ (respectively $2\pi/3$ and $\pi$). This implies that the limit \eqref{equation-limit} does not exist for $n=10$ and $n=12$. It is unlikely to exist for even $n>12$, but the limit will then involve an interplay between different Hecke eigenforms.
\end{remark}

In~\cite[Thm.~3.9]{BHLR22} it is shown that for fixed $g$ we have
\[\lim_{n\rightarrow\infty}\fraka_{2n}(\M_g)^{1/(2n)}=2g.\]
In the remainder of this section we prove a similar result for $\bb_{2n+1}(\M_g)$.
\begin{proposition}
\label{prop:bn}
    For fixed $g \geq 3$ one has
    \[
    \lim_{n\rightarrow\infty} \bb_{2n+1}(\M_g)^{1/(2n+1)}=2g.
    \]
\end{proposition}
\begin{proof}
Consider the functions $w_1$ and 
$f \colonequals  \frac{1}{6}w_1^3-\frac{1}{2}w_1w_2+\frac{1}{3}w_3-w_1$
on $X\colonequals [0,\pi]^g$. The maximum value of $\lvert w_1\rvert$ is attained at exactly two points in $X$,
namely the points $x\colonequals (0,\ldots,0)$ and $y \colonequals  (\pi,\ldots,\pi)$.
We have $w_1(x) = 2g$ and $w_1(y) = -2g$, and we also have 
$f(x) = (2/3)(2g^3-3g^2-2g) > 0$ and $f(y) = (-2/3)(2g^3-3g^2-2g) < 0$.

Let $V$ be the (open) subset of $X$ where $w_1 f > 0$,
so that $x$ and $y$ both lie in $V$, and let 
$W = X\setminus V$. Let $M$ be the supremum of $\lvert w_1\rvert$ on $W$, so that $M < 2g$.
For $\eps\in (0,2g-M)$ let $U_\eps$ be the subset of $X$ where 
$\lvert w_1 \rvert > 2g - \eps$, so that $U_\eps\subset V$, and
let $V_\eps = V\setminus U_\eps$.

Then, for every $n$, we have
\begin{align*}
\bb_{2n+1}(\M_g) 
 &= \int_X w_1^{2n+1} f \, d m_g \\
 &= \int_{U_\eps} w_1^{2n+1} f \, d m_g + \int_{V_\eps} w_1^{2n+1} f \, d m_g + \int_{W} 
 w_1^{2n+1} f \, d m_g \\
 &\ge \int_{U_\eps} w_1^{2n+1} f \, d m_g +  \int_W w_1^{2n+1} f \, d m_g \\
 &\ge (2g-\eps)^{2n+1} \int_{U_\eps} \lvert f\rvert \, d m_g - M^{2n+1} \int_W \lvert f\rvert\, d m_g,
 \end{align*}
where the third line follows from the fact that $w_1^{2n+1} f$ is positive on $V_\eps$ and
the fourth follows from the bounds on $\lvert w_1\rvert$ in $U_\eps$ and $W$.
Let $A\colonequals \int_{U_\eps} \lvert f\rvert \, d m_g$ and $B\colonequals\int_W \lvert f\rvert\, d m_g.$ Then
\[
\bb_{2n+1}(\M_g)^{1/(2n+1)}
  \ge (2g-\eps)\biggl(A - \Bigl(\frac{M}{2g-\eps}\Bigr)^{2n+1} B\biggr)^{1/(2n+1)},
\]
and the rightmost factor tends to $1$ as $n\to\infty$. Therefore,
$\liminf \bb_{2n+1}(\M_g)^{1/(2n+1)} \ge 2g.$

We also have
\begin{align*}
\bb_{2n+1}(\M_g) 
 &= \int_{U_\eps} w_1^{2n+1} f \, d m_g + \int_{X\setminus U_\eps} w_1^{2n+1} f \, d m_g\\
 &\le (2g)^{2n+1}\int_{U_\eps} \lvert f\rvert \, d m_g 
 +  (2g-\eps)^{2n+1} \int_{X\setminus U_\eps} \lvert f\rvert \, d m_g,
\end{align*}
so if we let $C \colonequals \int_X  \lvert f\rvert \, d m_g$ then
$\bb_{2n+1}(\M_g) \le  (2g)^{2n+1} A + (2g-\eps)^{2n+1} C,$ so
\[
\bb_{2n+1}(\M_g)^{1/(2n+1)}
  \le 2g \biggl(A + \Bigl(\frac{2g-\eps}{2g}\Bigr)^{2n+1} C\biggr)^{1/(2n+1)}.
\]
Once again the rightmost factor tends to $1$ as $n\to\infty$, so 
$\limsup \bb_{2n+1}(\M_g)^{1/(2n+1)} \le 2g,$
and the proposition is proven.
\end{proof}

\begin{remark}

Let $\X_g$ be either $\M_g$ or $\M_{g,1}$, where the latter denotes the moduli space of curves of genus $g$ together with a marked point. For any $k \geq 0$, $\lambda$ as in the proof of Theorem~\ref{theorem:H0H1}, 
and 
$g \geq \frac{3}{2}(k+1+|\lambda|)$, there is an isomorphism in Betti cohomology, $H^k(\X_g,\VV_{\lambda}) \cong H^k(\X_{g+1},\VV_{\lambda})$; see \cite[Thm.~1.1]{Looijenga} and \cite{Wahl}. These are called \emph{stable} cohomology groups. 

In \cite[Thm.~3.5.12]{BDPW} there is an alternative formula to that of \cite[Thm.~1.1]{Looijenga} for the dimensions of the stable cohomology groups of $\M_g$. Using this formula one can prove, in a way analogous to \cite[Thm.~7.0.2.]{BDPW}, that if $k < |\lambda|/3$ and $g \geq \frac{3}{2}(k+1+|\lambda|)$, then $H^k(\M_{g},\VV_{\lambda})=0$.
It follows that for each $k$ there are finitely many $\lambda$ for which $H^k(\M_g,\VV_{\lambda})$, with $g = \lceil \frac{3}{2}(k+1+|\lambda|) \rceil$, is non-zero. Again using \cite[Thm.~3.5.12]{BDPW} we find for instance that there are five such $\lambda$ for $k=2$ (see below) and fourteen such $\lambda$ for $k=3$. Note also that for $g \geq \frac{3}{2}(k+1+|\lambda|)$, $H^k(\M_{g},\VV_{\lambda})$ is zero if $k+|\lambda|$ is odd.

The result above also holds in $\ell$-adic cohomology. Moreover, 
every eigenvalue of Frobenius $F_q$ acting on the compactly supported $\ell$-adic cohomology group $H_c^{6g-6-k}(\M_g,\VV_{\lambda})$, for $g \geq \frac{3}{2}(k+1+|\lambda|)$, is equal to $q^{3g-3+(|\lambda|-k)/2}$; see for instance \cite{petersenetal2}. 

In forthcoming work by Miller, Patzt, Petersen, Randal-Williams entitled "Uniform twisted homological stability" it is shown that 
for $g \geq 3k+3$ (i.e. a bound that is independent of $\lambda$), there is an isomorphism in Betti cohomology, $H^k(\M_{g,1},\VV_{\lambda}) \cong H^k(\M_{g+1,1},\VV_{\lambda})$. 
It should be possible to show that this leads to an isomorphism $H^k(\M_{g},\VV_{\lambda}) \cong H^k(\M_{g+1},\VV_{\lambda})$ for all $g \geq g_{\rm stab}(k)$, with $g_{\rm stab}(k)$ a function that only depends upon $k$, cf. \cite{CM} and \cite[Rem.~3.5.11]{BDPW}. 
If we assume this to be true then we can combine the results above with the techniques in the proof of Theorem~\ref{theorem:H0H1} to conclude the following.

Let $d_{n,\lambda}$ denote the number of times the representation $V_{\lambda}$ appears in the $\USp_{2g}$-representation $V^{\otimes n}$. Fix any $K\geq 0$. Then,
for any $n \geq 1$ and $g \geq g_{\rm stab}(K)$, 
we have
\begin{equation} \label{eq-c}
\sum_{k=0}^K (-1)^k \, \cc_{k,n} \cdot q^{-k/2}=S_n(\M_g,q)/q^{3g-3+n/2}+O(q^{-(K+1)/2}), 
\end{equation}
where
\[
\cc_{k,n}=\sum_{\lambda}d_{n,\lambda} \cdot \dim H^k(\M_{g_{\rm stab}(k)},\VV_{\lambda}).
\]
From \cite[Thm.~3.5.12]{BDPW} we can for instance compute that 
\[
\cc_{2,n}=d_{n,(0)}+d_{n,(1^2)}+d_{n,(1^4)}+d_{n,(1^6)}+d_{n,(2^2,1^2)}. 
\]    
Note that by Theorem~\ref{theorem:H0H1}, $\cc_{0,n}=\fraka_n(\M_g)$ for $g \geq 2$, $\cc_{1,n}=\bb_n(\M_g)$ for $g \geq 3$ and Equation~\eqref{eq-c} holds with $g_{\rm stab}(0)=2$ and $g_{\rm stab}(1)=3$. 
\end{remark}


\section{Convergence of moments of the measures \texorpdfstring{$\mu_{q,g}$}{mu q g}} \label{sec:cvmoments}
Let $\M_g'(\F_q)$ be the set of $\F_q$-isomorphism classes of curves of genus $g>1$ over $\F_q$. If $g=1$, we abuse notation and let $\M_1=\M_{1,1}$ be the moduli space of elliptic curves and $\M'_1(\F_q)$ the set of $\F_q$-isomorphism classes of elliptic curves over $\F_q$.  Define a measure $\mu_{q,g}$ by
\[\mu_{q,g} \colonequals \frac{1}{\# \M_g(\F_q)} \sum_{C\in \M_g'(\F_q)} \frac{\delta_{\tau(C)}}{\#\Aut_{\F_q}(C)}\,,\]
where $\tau(C)\colonequals\Tr(C)/\sqrt{q}$ is the \emph{normalized trace} of $C$ and $\delta_{\tau(C)}$ is the Dirac $\delta$ measure supported at~$\tau(C)$.
We see that $\mu_{q,g}$ is a discrete probability measure on $I_g\colonequals [-2g,2g]$, since 
\begin{align*}
\mu_{q,g}(I_g) &= \frac{1}{\# \M_g(\F_q)} \sum_{C\in \M_g'(\F_q)} \frac{1}{\#\Aut_{\F_q}(C)} \\
&= \frac{1}{\# \M_g(\F_q)} \sum_{C\in \M_g(\F_q)} \underbrace{\sum_{C' \in \Twist(C)} \frac{1}{\#\Aut_{\F_q}(C)}}_{=\,1 \; \textup{by \cite[Prop.~5.1]{VdG92}}} = 1.
\end{align*}
We can introduce $\Nq_{q,g}(\tau)$  defined by 
\[
\Nq_{q,g}(\tau)\colonequals \frac{1}{\# \M_g(\F_q)}
\ \sum_{C \in \M_g'(\F_q), \tau(C)=\tau} 
\frac{1}{\#\Aut_{\F_q}(C)}
\]
and rewrite $\mu_{q,g} = \sum_{\tau \in I_g} \Nq_{q,g}(\tau) \delta_\tau$. 
Note that the definition of $\Nq_{q,g}(\tau)$ differs from the ones of \cite[App.~B]{LRRS} and \cite[\S4]{RRRSS}, in particular by a factor of $\sqrt{q}$ (this factor will appear again in Section~\ref{sec:experiment} but this definition is more natural for the measure).

From \cite[Rem.~3.5]{LachaudDis}, as a direct consequence of Katz--Sarnak results \cite[Thms.~10.7.12 and 10.8.2]{katz-sarnak}, there exists a probability measure $\mu_g \colon  I_g \to \R$ with a $\mathcal{C}^{\infty}$ density function $\Fd_g$ such that we have weak convergence of $\mu_{q,g}$ to $\mu_g$.
Writing
\[\Fd_g(\tau)=\int_{A_{\tau}} dm_g \; \textrm{\ with} \; A_{\tau}=\bigl\{(\theta_1,...,\theta_g)\in[0,\pi]^g:\,\textstyle\sum_{j}2\cos\theta_j = \tau\bigr\},\] we see this is equivalent to 
\begin{equation}\label{eq:eqweak}\lim_{q \to \infty} \int_{I_g} f \, d\mu_{q,g} = \int_{I_g} f(\tau)\Fd_g(\tau)\,d\tau
\end{equation}
 for all continuous functions $f\colon I_g \to \R$. 
 Moreover for all polynomial functions\footnote{In an earlier version and in \cite{BHLR22} following \cite[Cor.~4.3]{LachaudDis}, we wrote that this convergence rate holds for any continuous function. We cannot find a proof for this and prefer to state it now only for polynomial functions as Katz and Sarnark do. Fortunately this change has no consequence on the rest of \cite{BHLR22} : for instance Corollary~2.3 can be proven only using the pointwise convergence of the cumulative distributions which is equivalent to the weak convergence above.
} 
$P\colon I_g \to \R$
\begin{equation} \label{eq:weakconvergence}
\int_{I_g} P \, d\mu_{q,g} = \int_{I_g} P(\tau)\Fd_g(\tau)\,d\tau + O\left(\frac{1}{\sqrt{q}}\right).
\end{equation}
We will now find a refinement of \eqref{eq:weakconvergence} when $g \geq 2$. 
\begin{theorem} \label{th:betterKS}
Let
\begin{equation} \label{eq:Hd}
\Hd_g(\tau)=\int_{A_{\tau}}  \Bigl(\frac{1}{6}w_1^3-\frac{1}{2}w_1w_2+\frac{1}{3}w_3-w_1 \Bigr)\,  d m_g     
\end{equation}
be the function whose $n$th moments are equal to the numbers  $\bb_n(\M_g)$ given by the expression \eqref{eq:bn}. 
For $g\geq 2$ and every polynomial function $P : I_g \to \R$, we have
\begin{equation} 
\int_{I_g} P \, d\mu_{q,g} = \int_{I_g} P(\tau) \left(\Fd_g(\tau)-\frac{\Hd_g(\tau)}{\sqrt{q}}\right)\,d\tau + O\left(q^{-1}\right).
\end{equation}
\end{theorem}
\begin{proof}
Notice that 
\[
\frac{S_n(q,\M_g)}{\# \M_g(\F_q) \cdot q^{n/2}} = \int_{I_g} \tau^n \,d\mu_{q,g}.
\]
Using Deligne's theory of weights, as in the proof of Theorem~\ref{theorem:H0H1}, we find that 
\[
\#\M_g(\F_q)= \Tr \bigl(\mathrm{Fr}_q,H_c^{6g-6}(\M_g,\Q_{\ell})\bigr)+O\left(q^{3g-4}\right)=q^{3g-3}+O\left(q^{3g-4}\right),
\]
since $\M_g$ is irreducible of dimension $3g-3$. Hence 
\[
\frac{S_n(q,\M_g)}{\# \M_g(\F_q) \cdot q^{n/2}}= \frac{S_n(q,\M_g)}{q^{3g-3+n/2}}+ O(q^{-1}).\]
Using  Theorem~\ref{theorem:H0H1} \eqref{item:4} for $g \geq 2$, we then get
\begin{align*}
\int_{I_g} \tau^n \,d\mu_{q,g}
&= \frac{S_n(q,\M_g)}{\# \M_g(\F_q) \cdot q^{n/2}}   \\
&= \frac{S_n(q,\M_g)}{q^{3g-3+n/2}}+ O(q^{-1})\\ 
&= \fraka_n(\M_g) - \frac{\bb_n(\M_g)}{\sqrt{q}} + O(q^{-1}) \\
&= \int_{I_g} \tau^n \left(\Fd_g(\tau)-\frac{\Hd_g(\tau)}{\sqrt{q}}\right)\,d\tau + O\left(q^{-1}\right).
\end{align*}
\end{proof}


\section{The elliptic and hyperelliptic cases: results and experiments}
\label{sec:experiment}

Katz--Sarnak results show that for every interval $J\subseteq I_g$, the probability that a
random curve of genus $g$ over $\F_q$ (or a random  hyperelliptic curve of genus $g$ over $\F_q$) has
normalized trace in $J$ tends towards a fixed value as $q\to\infty$, this value being 
$\int_J \Fd_g(\tau) \,d\tau$, where $\Fd_g$ is the density function for the measure $\mu_g$
defined at the beginning of Section~\ref{sec:cvmoments}.  Here the interval $J$ is fixed, and
we let $q$ tend to infinity. One can wonder how rapid this convergence is.
For instance, suppose the interval $J$ has length~$x$. How large must $q$ become in order for
the actual probability that a normalized trace lies in $J$ is well-approximated  by the Katz--Sarnak prediction? Could it even be the case that the approximation is reasonably good when $q$ is as large as $1/x^2$, 
so that $x\approx 1/\sqrt{q}$ and there is exactly one integer $t$ with $t/\sqrt{q}\in J$?
In other words, can we use the Katz--Sarnak distribution to estimate the number of curves over
$\F_q$ with a given trace? Since the measures $\mu_{q,g}$ converge weakly to $\mu_g$, 
one might hope that for every $\tau\in I_g$,
the integral of $\mu_{q,g}$ over an interval of length $1/\sqrt{q}$ containing $\tau$
would be close to the integral of $\mu_g$ over this interval. If we let $t$ be the unique integer
such that $t/\sqrt{q}$ is contained in this interval, this optimistic approximation then translates to
\[
\sqrt{q}\,\N_{q,g}\biggl(\frac{t}{\sqrt{q}}\biggr) \approx 
 \Fd_g\biggl(\frac{t}{\sqrt{q}}\biggr).
\]
Since $\N_{q,g}(t/\sqrt{q})$
gives us the weighted number of curves with trace~$t$, if this approximation is close to the truth 
we would have a good estimate for the number of such curves.

\begin{remark}
We do not know how to prove that this estimate holds, and indeed we will
see below that it does \emph{not} hold, without modification, for 
hyperelliptic curves. One consequence of this estimate, however, is the
much weaker statement that for every fixed value of~$t$, the value of
$\N_{q,g}(t)$ converges to $0$ as $q$ increases. It is at least easy to
show that this weaker statement holds for $t=0$, by the following argument.

Given $\eps>0$, let $f\colon I_g\to [0,1]$ be a continuous function 
with $f(0) = 1$ and with $f(\tau)=0$ when $\lvert \tau\rvert \ge \eps$.
From \eqref{eq:eqweak} we find that for $q$ large enough we have
\[
\biggl\lvert\,\int_{I_g} f \, d\mu_{q,g} - \int_{I_g} f(\tau) \Fd_g(\tau) \, d\mu_{g}\biggr\rvert \leq \eps.\] 
Hence  
\[
0 \leq \N_{q,g}(0) \leq \int_{I_g} f \, d\mu_{q,g} \leq 
   \int_{\lvert\tau\rvert<\eps} f(\tau) \Fd_g(\tau)\,d\tau +\eps \leq (2\left\|\Fd_g\right\|_\infty+1) \eps. 
\]

\end{remark}

As we intimated in the preceding remark, for hyperelliptic curves we can prove
that the na\"{\i}ve approximation for $\N_{q,g}$ described above cannot hold.
To state our result precisely, we introduce a function $\N^\hyp_{q,g}(\tau)$, 
which we define analogously to how we defined~$\N_{q,g}(\tau)$:
\[
\N^\hyp_{q,g}(\tau) \colonequals \frac{1}{\#\calH_g(\F_q)} 
            \sum_{\substack{C\in\calH_g'(\F_q)\\ \tau(C) = \tau}}
            \frac{1}{\#\Aut(C)}.
\]
Here by $\calH_g(\F_q)$ we mean the set of $\Fbar_q$-isomorphism classes of hyperelliptic curves 
of genus $g$ over $\F_q$, and by $\calH_g'(\F_q)$ we mean the set of $\F_q$-isomorphism classes 
of such curves. Note that for an integer $t$ in~$I_g$, the value  $q^{2g-1} \N^\hyp_{q,g}(t/\sqrt{q})$ is
then the weighted number of genus-$g$ hyperelliptic curves over $\F_q$ with trace~$t$.

\begin{proposition}
\label{P:nolimit}
Fix $g>1$ and $\eps\in[0,2g)$,  let $r_g \colonequals \sum_{i=0}^{2g+2} (-2)^i/i!$, 
and let $v = \int_{2g-\eps}^{2g} \Fd_g(\tau) \,d\tau$.
Suppose there are constants $b_g\le c_g$ such that for
every sufficiently large prime power $q$ and for every integer $t$ in
$[-(2g-\eps)\sqrt{q}, (2g-\eps)\sqrt{q}\,]$, we have
\[
\frac{b_g}{\sqrt{q}} \Fd_g\biggl(\frac{t}{\sqrt{q}}\biggr) 
  \le \N^\hyp_{q,g}\biggl(\frac{t}{\sqrt{q}}\biggr)
  \le \frac{c_g}{\sqrt{q}}  \Fd_g\biggl(\frac{t}{\sqrt{q}}\biggr).
\]
Then $b_g\le (1-r_g)/(1-2v)$ and $c_g\ge (1 + r_g - 4v)/(1-2v)$.
\end{proposition}

The proof is based on the following lemma. 
\begin{lemma}
\label{L:evenodd}
Fix $g>1$, and let $r_g$ be as in Proposition~\textup{\ref{P:nolimit}}. If $q$ is 
an odd prime power then
\[
\sum_{t\,\even} \N^\hyp_{q,g}\biggl(\frac{t}{\sqrt{q}}\biggr) 
= \frac{1+r_g}{2} + O\Bigl(\frac1q\Bigr) 
\text{\quad and\quad}
\sum_{t\,\odd} \N^\hyp_{q,g}\biggl(\frac{t}{\sqrt{q}}\biggr) = \frac{1-r_g}{2} + O\Bigl(\frac1q\Bigr).
\]
\end{lemma}

\begin{proof}
Fix an odd prime power $q$, fix a nonsquare $n\in\F_q$, and consider the set $H$
consisting of all pairs $(c,f)$, where $c\in\{1,n\}$ and $f\in\F_q[x]$ is a
monic separable polynomial of degree $2g+1$ or $2g+2$. A result of
Carlitz~\cite[\S6]{Carlitz1935} shows that $\#H = 2q^{2g+2} - 2q^{2g}.$
The group $\PGL_2(\F_q)$ acts on $H$: Given a matrix
$[\begin{smallmatrix}r&s\\t&u\end{smallmatrix}]$ and an element $(c,f)$ of $H$,
let $(d,g)$ be the unique element of $H$ such that
\[
d g(x) = c e^2 (tx+u)^{2g+2} f\Bigl(\frac{rx+s}{tx+u}\Bigr)
\]
for some $e\in\F_q^\times.$ Note that the stabilizer of $(c,f)$ is isomorphic to
the reduced automorphism group $\RedAut(C)$ of the hyperelliptic curve $C\colon y^2 = cf$, that is,
the quotient of the full automorphism group of $C$ by the subgroup
generated by the hyperelliptic involution.

The map $\gamma$ that sends $(c,f)\in H$ to the hyperelliptic curve
$y^2 = c f$ takes $H$ onto $\calH'_g(\F_q)$. Given a curve $C\in\calH'_g(\F_q)$,
let $(c,f)\in H$ be such that $\gamma((c,f)) = C$. Then
\[
\#(\PGL_2(\F_q)\cdot (c,f)) = \frac{\#\PGL_2(\F_q)}{\#\RedAut(C)},
\]
so that 
\begin{equation}
\label{EQ:weightaction}
\frac{ \#\gamma^{-1}(C)}{\#\PGL_2(\F_q)} = \frac{1}{\#\RedAut(C)} = \frac{2}{\#\Aut(C)}.
\end{equation}

Let $H_\even$ be the subset of $H$ consisting of the pairs $(c,f)$ such that
the curve $\gamma(c,f)$ has even trace. Let $H'_\even$ be the subset of $H$
consisting of the pairs $(c,f)$ such that $f$ has degree $2g+2$ and has an even
number of roots. Then $H'_\even\subseteq H_\even$, and 
$H_\even\setminus H'_\even$ consists of pairs $(c,f)\in H_\even$ such that $f$
{has degree $2g+1$}. Therefore
\[
\bigl\vert \#H_\even - \#H'_\even\bigr\vert
\le 2q^{2g+1}.
\]

Leont{\cprime}ev~\cite[Lem.~4, p.~302]{Leontev2006english} gives the generating
function for the number of (not necessarily separable) monic polynomials of a
fixed degree over $\F_q$ that have a given number of roots. To find the number
of such polynomials with an even number of roots, we simply need to take the
average of the values of this generating function evaluated at $-1$ and at~$1$.
We find that
\[
\#\left\{\vcenter{\hsize=2.0in\noindent 
monic polynomials of degree $2g+2$\\
over $\F_q$ with an even number of roots}
\right\}
=  \frac{1+r_g}{2} q^{2g+2} + O(q^{2g+1}).
\]
The result of Carlitz mentioned earlier shows that 
\[
\#\left\{\vcenter{\hsize=1.9in\noindent 
non-separable monic polynomials\\
of degree $2g+2$ over $\F_q$}\right\} =q^{2g+1}.
\]
Therefore $\# H'_\even =  (1+r_g) q^{2g+2} + O(q^{2g+1})$, so that
$\# H_\even =  (1+r_g) q^{2g+2} + O(q^{2g+1})$ as well.

Using~\eqref{EQ:weightaction} we see that 
\begin{align*}
\sum_{t \,\even} \N^\hyp_{q,g}\biggl(\frac{t}{\sqrt{q}}\biggr)
&= \frac{1}{\#\calH_g(\F_q)} \sum_{\substack{C\in \calH'_g(\F_q)\\ \Tr(C)\, \even}} 
\frac{1}{\#\Aut_{\F_q}(C)}\\
&= \frac{1}{\#\calH_g(\F_q)} \sum_{\substack{C\in \calH'_g(\F_q)\\ \Tr(C)\, \even}} \frac{\#\gamma^{-1}(C)}{2\#\PGL_2(\F_q)}\\
&= \frac{1}{2\#\calH_g(\F_q)\#\PGL_2(\F_q)} \#H_\even\\
&=\frac{1}{2 q^{2g-1} (q^3-q)}\bigl((1+r_g)q^{2g+2} + O(q^{2g+1})\bigr)\\
&=\frac{1+r_g}{2} + O\Bigl(\frac{1}{q}\Bigr).
\end{align*}
This gives us the first equality in the conclusion of the lemma. The second
follows analogously.
\end{proof}

\begin{proof}[Proof of Proposition~\textup{\ref{P:nolimit}}]
Suppose the hypothesis of the proposition holds for a given $g$ and $\eps$.
For a given $q$, we let 
$m = \lfloor 2\sqrt{q} \rfloor$ and we consider several subintervals 
of~$[-2g\sqrt{q},2g\sqrt{q}]$: 
\begin{alignat*}{3}
J_0 &\colonequals \bigl[-mg,mg\bigr] & &\qquad&
J_2 &\colonequals \bigl[-2g\sqrt{q},-(2g-\eps)\sqrt{q}\bigr)\\
J_1 &\colonequals \bigl[-(2g-\eps)\sqrt{q}, (2g-\eps)\sqrt{q}\,\bigr] & &\qquad&
J_3 &\colonequals \bigl((2g-\eps)\sqrt{q},2g\sqrt{q}\,\bigr].
\end{alignat*}
Now we interpret the sum
\[
S_\even \colonequals 
\sum_{t \ \even} \N^\hyp_{q,g}\biggl(\frac{t}{\sqrt{q}}\biggr) 
\]
in two ways. On the one hand, from Lemma~\ref{L:evenodd} we have
\[
S_\even = \biggl(\frac{1+r_g}{2}\biggr)
    + O\Bigl(\frac{1}{q}\Bigr)\,. 
\]
On the other hand, for $q$ large enough we have
\begin{align}
S_\even
\notag &=  \sum_{\substack{t\in J_1\\ t \ \even}} \N^\hyp_{q,g}\biggl(\frac{t}{\sqrt{q}}\biggr)
  + \sum_{\substack{t\in J_2\\ t \ \even}} \N^\hyp_{q,g}\biggl(\frac{t}{\sqrt{q}}\biggr)
  + \sum_{\substack{t\in J_3\\ t \ \even}} \N^\hyp_{q,g}\biggl(\frac{t}{\sqrt{q}}\biggr)\\
\notag &=  \sum_{\substack{t\in J_1\\ t \ \even}} \N^\hyp_{q,g}\biggl(\frac{t}{\sqrt{q}}\biggr)
  + 2\sum_{\substack{t\in J_3\\ t \ \even}} \N^\hyp_{q,g}\biggl(\frac{t}{\sqrt{q}}\biggr)\\
\label{3terms} & \le 
  \frac{c_g}{2} \sum_{\substack{t\in J_1\\ t \ \even}} 
      \Fd_g\biggl(\frac{t}{\sqrt{q}}\biggr) \biggl(\frac{2}{\sqrt{q}}\biggr)
+ 2\sum_{t\in J_3}   \N^\hyp_{q,g}\biggl(\frac{t}{\sqrt{q}}\biggr)
\,.
\end{align}
The first sum in \eqref{3terms} is a Riemann sum for the integral of $\Fd_g(\tau) \, d\tau$
over the interval $[-2g+\eps, 2g-\eps]$, so as $q\to \infty$ the first term
in \eqref{3terms} approaches $c_g(1-2v)/2$. The second sum is the measure, 
with respect to $\mu_{q,g}$, of the interval $[2g-\eps,2g]$. Since the 
$\mu_{q,g}$ converge weakly to $\mu_g$, the second term of~\eqref{3terms}
approaches $2v$ as $q\to\infty$.

Combining these two interpretations of $S_\even$, we find that
\[
\biggl(\frac{1+r_g}{2}\biggr) \le \frac{c_g(1-2v)}{2} + 2v
\]
so that $c_g \ge (1 + r_g - 4v)/(1-2v)$.

Similarly, we can consider the sum
\[
S_\odd \colonequals 
\sum_{t \ \odd} \N^\hyp_{q,g}\biggl(\frac{t}{\sqrt{q}}\biggr).
\]
From Lemma~\ref{L:evenodd} we see that
\[
S_\odd = \biggl(\frac{1-r_g}{2}\biggr)
    + O\Bigl(\frac{1}{q}\Bigr)\,. 
\]
But we also have
\[
S_\odd 
\ge \frac{b_g}{2} \sum_{\substack{t\in J_1\\ t \ \odd}} 
      \Fd_g\biggl(\frac{t}{\sqrt{q}}\biggr) \biggl(\frac{2}{\sqrt{q}}\biggr),
\]
and the expression on the right approaches $b_g(1-2v)/2$ as $q\to\infty$. 
This shows that
\[
\biggl(\frac{1-r_g}{2}\biggr) \ge \frac{b_g(1-2v)}{2},
\]
so we find that $b_g\le (1-r_g)/(1-2v)$.
\end{proof}

\begin{remark}
In the statement of Proposition~\ref{P:nolimit}, we only assume that 
the condition on $\N^\hyp_{q,g}(t/\sqrt{q})$ holds for $t$ more than $\eps\sqrt{q}$ away from the 
ends of the interval $[-2g\sqrt{q}, 2g\sqrt{q}\,]$ because when
$\lvert t\rvert > g\lfloor2\sqrt{q}\rfloor$ we have $\N^\hyp_{q,g}(t/\sqrt{q}) = 0$.
If we did not exclude the tail ends of the interval, the hypothesis of the
proposition would only hold if we took $b_g=0$, which is not an interesting
approximation.
\end{remark}

\input genus2-1009.tex

Figure~\ref{fig:genus2}  shows the 
value of $\N^\hyp_{q,g}(t/\sqrt{q})$ for all integers $t\in[-4\sqrt{q},4\sqrt{q}]$,
where $q = 1009$, together with the density
function $\Fd_2$ for the limiting Katz--Sarnak measure, scaled by the two
factors $b=38/45$ and $c = 52/45$ given by Proposition~\ref{P:nolimit} 
for $g=2$ and $\eps=0$.

\input genus2-scaled.tex

The key to Proposition~\ref{P:nolimit} is the imbalance between the
likelihood of even versus odd traces for hyperelliptic curves. The 
obvious workaround would be to scale the counts for the even and odd
traces by the factors given in the proposition for $\eps=0$. One can
ask whether the scaled curve counts then better match the limiting 
Katz--Sarnak distribution. Figure \ref{fig:genus2scaled} suggests that
perhaps this parity factor is the main obstruction to obtaining decent
estimates from the na\"{\i}ve Katz--Sarnak approximation.\\

The proof of Proposition~\ref{P:nolimit} carries through for elliptic
curves exactly as it does for hyperelliptic curves of a given genus $g>1$.
We do not include genus-$1$ curves in the statement of the proposition,
however, because as we will see in Proposition~\ref{P:nolimit1},
for $g=1$ there is no value of $c_1$ that
satisfies the hypothesis of the proposition when $\eps\le 1$, 
while the conclusion of the proposition is trivial when $\eps>1$
because the resulting upper bound on $b_1$ will be greater than $1$ and the lower bound
on $c_1$ will be less than~$1$.

When $g=1$, the density function of the limiting Katz--Sarnak measure 
on $I_1$ is $\Fd_1=  (2\pi)^{-1}\sqrt{4-\tau^2}$.
Let $N_{q,t}$ denote the weighted number of elliptic curves over $\F_q$ with trace $t$. 
For some values of $t$ in $[-2\sqrt{q},2\sqrt{q}\,]$ we have $N_{q,t} = 0$; in addition to those
$t$ with $\lvert t\rvert > \lfloor 2\sqrt{q}\rfloor$, this happens for 
most values of $t$ that are not coprime to~$q$. But even if we exclude these
values, and even if we restrict attention to values of $t$ that are near the
center of the interval $[-2\sqrt{q},2\sqrt{q}\,]$, the following proposition
shows that we cannot hope to approximate $N_{q,t}$ by the quantity
\[q^{1/2} \Fd_1\biggl(\frac{t}{\sqrt{q}}\biggr) = \frac{1}{2\pi}\sqrt{4q - t^2}\,.\]

\begin{proposition}
\label{P:nolimit1}
For every $c>0$, there are infinitely many values of $q$ and $t$
such that $\lvert t \rvert \le \sqrt{q}$ and $N_{q,t}> c \sqrt{4q-t^2}$.
\end{proposition}

\begin{proof}
Let $\Delta_0$ be a fundamental quadratic discriminant with $\Delta_0<-4$
and let $\chi$ be the quadratic character modulo~$\Delta_0$. For a
given value of $n$, let $f$ be the product of the first $n$ primes $p$ 
that are inert in $\Q(\sqrt{\Delta_0})$.
Since the product over all inert primes of $1+1/p$ diverges 
(see~\cite[Lem.~1.14]{Cox2013} and \cite[Exer.~6, p.~176]{Apostol1976}),
when $n$ is large enough we have
\[
\prod_{p\mid f} \biggl(1 + \frac{1}{p}\biggr) >
\frac{c\pi^2}{3} \frac{\sqrt{\lvert\Delta_0\rvert}}{h(\Delta_0)}\,.
\]
Choose $n$ so that this
holds, and let $q_0$ be a prime of the form $x^2 - f^2\Delta_0 y^2$, where 
$x$ and $y$ are positive integers. 
Note that $x$ must be coprime to $q_0$ because $0<x<q_0$. 
Let $\varpi = x + fy\sqrt{\Delta_0}$, viewed as
an element of the upper half plane. Since $x$ is coprime to $q_0$, 
$\varpi$ is the Weil number of
an isogeny class of ordinary elliptic curves over $\F_{q_0}$. 

Let $\theta$ be the argument of~$\varpi$ and let $m$ be the smallest integer
such that $\pi/3\le m\theta <  2\pi/3$.
Write
$\varpi^m = u + fv\sqrt{\Delta}$ for integers $u$ and $v$, let 
$q = q_0^m = u^2 - f^2v^2\Delta$, and 
let $t = 2u$. Then $\varpi^m$ is the Weil number for an isogeny class $\I$ of
ordinary elliptic curves over $\F_q$, and the trace of this isogeny class
is~$t$. We have $\lvert t\rvert \le \sqrt{q}$ because the argument of
$\varpi^m$ lies between $\pi/3$ and $2\pi/3$.

The number of elliptic curves in the isogeny class $\I$ is equal to
the Kronecker class number $H(\Delta)$ of the discriminant 
$\Delta\colonequals t^2 - 4q = 4f^2v^2\Delta_0$. 
By~\cite[p.~696]{Howe2022}
we have
\[
H(\Delta) = h(\Delta_0) \prod_{p^e\parallel F}
\Bigl(1 + \Bigl(1 - {\textstyle\frac{\chi(p)}{p}}\Bigr)
           (p + \cdots + p^e)\Bigr)\,,
\]
where $F = 2fv$, so 
\[
\frac{H(\Delta)}{\sqrt{4q-t^2}}
= \frac{h(\Delta_0)}{\sqrt{\lvert\Delta_0\rvert}} 
\prod_{p^e\parallel F}
\Bigl(p^{-e} + \Bigl(1 - {\textstyle\frac{\chi(p)}{p}}\Bigr)
           (1 +  p^{-1} + \cdots + p^{1-e})\Bigr)\,.
\]
Now,
\[
p^{-e} + \Bigl(1 - {\textstyle\frac{\chi(p)}{p}}\Bigr)
           (1 +  p^{-1} + \cdots + p^{1-e})
\ge
\begin{cases}
1 + 1/p   & \text{if $\chi(p) = -1$;}\\
1 - 1/p^2 & \text{if $\chi(p) \neq -1$,}
\end{cases}
\]
so we have
\begin{align*}
\frac{H(\Delta)}{\sqrt{4q-t^2}}
&\ge \frac{h(\Delta_0)}{\sqrt{\lvert\Delta_0\rvert}} 
 \prod_{\substack{p \mid F\\ \chi(p)=-1}}\Bigl(1 + \frac{1}{p}\Bigr)
 \prod_{\substack{p \mid F\\ \chi(p)\ne -1}}\Bigl(1 - \frac{1}{p^2}\Bigr)\\
&\ge  \frac{h(\Delta_0)}{\sqrt{\lvert\Delta_0\rvert}} 
 \prod_{p \mid f}\Bigl(1 + \frac{1}{p}\Bigr)
 \prod_{p}\Bigl(1 - \frac{1}{p^2}\Bigr)\\
&\ge  \frac{h(\Delta_0)}{\sqrt{\lvert\Delta_0\rvert}} 
 \biggl(\frac{c\pi^2}{3} \frac{\sqrt{\lvert\Delta_0\rvert}}{h(\Delta_0)}\biggr)
 \biggl(\frac{6}{\pi^2}\biggr)\\
&\ge 2c. 
\end{align*}
Since the curves in $\I$ are ordinary and the discriminants of their
endomorphism rings are neither $-3$ nor~$-4$, they all have automorphism
groups of order~$2$, so $N_{q,t} = H(\Delta)/2$. It follows that
\[
N_{q,t} \ge c \sqrt{4q-t^2},
\]
as claimed.
\end{proof}

Figure~\ref{fig:genus1} shows the weighted number of elliptic curves
over $\F_{1000003}$ of each possible trace, as well as the
limiting density function $\Fd_1(\tau) = (2/\pi)\sqrt{4-\tau^2}$.
We see that the plotted points do not appear to be near the
density function.

\input genus1-1000003.tex


\section{The non-hyperelliptic case: experiments and conjectures} \label{sec:asymmetry}

We consider now the case of non-hyperelliptic curves of genus~$g=3$. For this purpose, for $g\geq 3$ we introduce the function $\N^\nhyp_{q,g}(\tau)$, which we define
analogously to how we defined~$\N_{q,g}(\tau)$ and $\N^\hyp_{q,g}(\tau)$:
\[
\N^\nhyp_{q,g}(\tau) \colonequals \frac{1}{\#\Mnh_g(\F_q)} 
            \sum_{\substack{C\in {\Mnh_g}'(\F_q)\\ \tau(C) = \tau}}
            \frac{1}{\#\Aut(C)}.
\]
Here by $\Mnh_g(\F_q)$ we mean the set of $\Fbar_q$-isomorphism classes of non-hyperelliptic curves 
of genus $g$ over $\F_q$, and by ${\Mnh_g}'(\F_q)$ we mean the set of $\F_q$-isomorphism classes 
of such curves. The associated measures will still weakly converge to the measure $\mu_g$ with density $\Fd_g$. But experimentally, the behavior looks much smoother than in the elliptic or hyperelliptic cases as illustrated by
Figure~\ref{fig:genus3} 
for $g=3$ and $q=53$.\footnote{When using the data of \cite{LRRS} to draw this figure, we noticed that there were some errors in the code when computing the automorphism group of twists for small dimensional strata, giving 728 extra `weighted' curves. This is a very small proportion with respect to $53^6+1$ curves and does not affect the general shape of the curve.} Note that a similar behavior would certainly hold considering all curves of genus $3$.
Heuristically, these patterns could be understood as an averaging for a given trace over several isogeny classes but this idea does not work for the hyperelliptic locus as we have seen in Section~\ref{sec:experiment} and something more is needed for a family of curves to `behave nicely.'
Still, the experimental data in genus $3$ lead us to state the following conjecture.

\input genus3-53.tex

\begin{conjecture} \label{conj:ponctual-conv}
Let $g \geq 3$. For all $\tau\in I_g$, for all $\eps>0$ and for all large enough $q$, there exists $t\in\NN$ such that $\lvert\tau-t/\sqrt{q}\rvert<1/(2\sqrt{q})$ and $\lvert\sqrt{q} \cdot  \Nq^\nhyp_{q,g}(t/\sqrt{q})- \Fd_g(t/\sqrt{q})\rvert<\eps$. \end{conjecture}

Another way to phrase this conjecture is to replace the measure $\mu_{q,g}$ by a measure with density given by the histogram with height $\sqrt{q} \cdot \Nq^\nhyp_{q,g}(t/\sqrt{q})$ and base centered at $t/\sqrt{q}$ of length $1/\sqrt{q}$ for all $t\in [-2g \sqrt{q},2g \sqrt{q}]$. The conjecture asserts that the densities of these measures converge to the density $\Fd_g$ at each point of $I_g$. This is stronger than weak convergence of the measures \cite{scheffe}.  \\

 We now conclude by looking at the symmetry breaking for the trace distribution of (non-hyperelliptic) genus 3 curves. In general, if $C$ is a hyperelliptic curve of genus $g$ over $\F_q$ with trace $t$, then its quadratic twist for the hyperelliptic involution has trace $-t$ and therefore the distribution of the number of hyperelliptic curves of genus $g$ over $\F_q$ as a function of their trace is  symmetric. For non-hyperelliptic curves, the distribution has no reason to be symmetric anymore. Actually, if a principally polarized abelian variety over $\F_q$ is the Jacobian (over $\F_q$) of a non-hyperelliptic curve, then its quadratic twist is never a Jacobian. This obstruction, known as \emph{Serre's obstruction}, is a huge obstacle to finding a closed formula for the maximal number of rational points for $g=3$ \cite{lauterg3}, whereas such formulas are  known for $g=1$ \cite{deuring} and $g=2$ \cite{serre-point}. Although we cannot improve on the state-of-art of this question, we can study this asymmetry with the probabilistic angle and the results we got before.

To visualize this asymmetry, let us consider the signed measure $\nu_{q,g} = \mu_{q,g} - (-1)^* \mu_{q,g}$ where $(-1)^* \mu_{q,g}$ is the discrete image signed measure defined by 
\[(-1)^* \mu_{q,g} =\frac{1}{\# \M_g(\F_q)} \sum_{C\in \M_g'(\F_q)} \frac{\delta_{-\tau(C)}}{\#\Aut_{\F_q}(C)}.\]
We get the following consequence of Theorem~\ref{theorem:H0H1}.

\begin{proposition} \label{prop:vweak}
The sequence of signed measures $(\nu_{q,g})$ weakly converges to the $0$ measure.
\end{proposition}
\begin{proof}
By definition, the even moments of $\nu_{q,g}$ are zero. By Theorem~\ref{theorem:H0H1} the odd moments of $\sqrt{q}\, \nu_{q,g}$ are equal to 
\[
2 \frac{S_n(q,\M_g)}{q^{3g-3+(n-1)/2}} 
= -2 \bb_n(\M_g) + O\left(\frac{1}{\sqrt{q}}\right).
\]
Hence all moments of $\nu_{q,g}$ are $0$. 
Now if $f$ is any continuous function on the compact interval $I_g=[-2g,2g]$, then by the Stone--Weierstrass theorem, for every $\eps>0$  we can find a polynomial $P$ such that $\lvert f(\tau)-P(\tau)\rvert\leq \eps$ for all $\tau\in I_g$. Therefore we have
\[
\biggl\lvert \int_{I_g} f \,d\nu_{q,g} \biggr\rvert \leq \biggl\lvert\int_{I_g} (f-P) \,d\nu_{q,g} + \int_{I_g} P \,d\nu_{q,g}\biggr\rvert \leq \eps \|\nu_{q,g}\| + \biggl\lvert \int_{I_g} P \,d\nu_{q,g}\biggr\rvert.\]
The last term is a sum of moments which converges to $0$ when $q$ goes to infinity. The variation of $\nu_{g,q}$ is also uniformly bounded since
\[\|\nu_{q,g}\| = \lvert\nu_{q,g}\rvert(I_g)=\sum_{\tau} \,\Bigl\lvert\Nq_{q,g}(\tau)-\Nq_{q,g}(-\tau)\Bigr\rvert \leq 2 \sum_{\tau} \Nq_{q,g}(\tau) = 2 \mu_{q,g}(I_g)=2.\]
\end{proof}

Having a $0$ measure is not very interesting and the proof of Proposition~\ref{prop:vweak} shows that it would be much more interesting to study the weak convergence of the sequence of signed measures $(\sqrt{q} \,\nu_{q,g})$. We have from the previous proof the following corollary.
\begin{corollary} \label{cor:momentsnu}
The even moments of $\sqrt{q} \,\nu_{q,g}$ are zero and the odd $n$th moments of the sequence $(\sqrt{q}\, \nu_{q,g})$ converge to $-2 \bb_n(\M_g)$. 
\end{corollary}
Unfortunately we cannot prove weak convergence: the rest of the proof fails as we do not know if one can bound $\sqrt{q}\, \|\nu_{q,g}\|$ uniformly in $q$ (which is a necessary condition for weak convergence). Moreover, one cannot expect a general result from the convergence of moments alone as in the case of (positive) measures as the following counterexample shows.
\begin{example} \label{ex:moments-signed}
Consider the sequence of signed measures $(\mu_i)$ with density $i \sin i x$ on the interval $[0, 2\pi]$. The sequence of $n$th moments converges to $-(2\pi)^n$ which is the $n$th moment of the signed measure $\mu=-\delta_{2 \pi}$. But $\|\mu_i\|=4i$ which is not bounded and therefore the sequence $(\mu_i)$ does not weakly converge (to $\mu$); see for instance \cite[Prop.~1.4.7]{bogachev}.    
\end{example}

Recall from \eqref{eq:Hd} that the $n$th moment of the function
\[
\Hd_g(\tau)=\int_{A_{\tau}}  \Bigl(\frac{1}{6}w_1^3-\frac{1}{2}w_1w_2+\frac{1}{3}w_3-w_1 \Bigr)\,  d m_g, 
\]
with $A_{\tau}=\{(\theta_1,\dots,\theta_g)\in[0,\pi]^g:\,\sum_{j}2\cos\theta_j= \tau\}$, is 
equal to $\bb_n(\M_g)$.
Because of the convergence of the moments above, we conjecture the following.
\begin{conjecture} \label{conj:diff} 
For $g \geq 3$, the sequence of signed measures $(\sqrt{q} \,\nu_{q,g})$ weakly converges to the continuous signed measure with density $-2 \Hd_g$. 
\end{conjecture}
Such a result would for instance imply that $\sqrt{q}\, \|\nu_{q,g}\|$ is uniformly bounded, hence there exists a constant $C>0$  such that for all $q$ and all $\tau=t/\sqrt{q}$,  we have $\lvert\Nq_{q,g}(\tau)-\Nq_{q,g}(-\tau)\rvert \leq C/\sqrt{q}$. 
   
In genus $3$, in the same spirit as in Section~\ref{sec:experiment}, one can run experiments which illustrate how the values 
\[
\left\{q \, \left(\Nq_{q,g}\left(\frac{t}{\sqrt{q}}\right)-\Nq_{q,g}\left(\frac{-t}{\sqrt{q}}\right)\right)\right\}_{0 \leq t \leq g \lfloor 2 \sqrt{q}\rfloor}
\]  
are close to the values $-2\Hd_3(t/\sqrt{q})$. See for instance Fig.~\ref{fig:comp} for $q=53$. Seeing the data, one may even wonder if something stronger would hold in the same line as Conjecture~\ref{conj:ponctual-conv}, at least for $g=3$.\\

Under this conjecture, one can use the moments of the density function $\Hd_3$ to revisit the result of 
\cite{RRRSS}.
Based on results of \cite{bucur}, the authors gave a heuristic explanation for the distribution of the points 
\[
p_{t,q}=\left(\frac{t}{\sqrt{q}},q \, \left(\Nq_{q,g}\left(\frac{t}{\sqrt{q}}\right)-\Nq_{q,g}\left(\frac{-t}{\sqrt{q}}\right)\right)\right)
\]
when $0 \leq t \leq g \lfloor 2\sqrt{q}\rfloor$ by comparing it with the distribution of differences around the mean in the binomial law \cite[Cor.~2.3]{RRRSS}. With the arguments given there, the distribution is approximated by the function 
\[
\calV^{\lim}(\tau) =  \tau (1-\tau^2/3) \cdot \left(\frac{1}{\sqrt{2\pi}} e^{-\tau^2 / 2}\right).
\]
Graphically for $q=53$, the comparison looks acceptable but not perfect (see Fig.~\ref{fig:comp}). This is fair as the heuristic grew from a result true when the degree of the plane curves in play is larger than $2q-1$. As presently we are dealing with non-hyperelliptic curves of genus~$3$, represented as  plane curves of degree~$4$, the condition is obviously never fulfilled. It is therefore already stunning that a close, albeit imperfect, match was found in this way.

We now take a different road based on Conjecture~\ref{conj:diff} and approximate the density $-2\Hd_3$ by a function $\nu^{\lim}$ using the moments $\bb_n(\M_3)$. 
By Theorem~\ref{theorem:H0H1}, they can be efficiently computed using any symmetric polynomial package. We used Maple and the package SF~\cite{SF} to compute $\bb_{n}(\M_3)$ for $n=1,3,5,\ldots,25$, and found the following values:

\bigskip
\begin{center}
\begin{tabular}{crrcrrcr}
\toprule
$n$ & $\bb_{n}(\M_3)$ &\hbox to 2em{}& $n$ & $\bb_{n}(\M_3)$ &\hbox to 2em{}& $n$ & $\bb_{n}(\M_3)$ \\
\cmidrule{1-2}\cmidrule{4-5}\cmidrule{7-8}
1 &    0 && 11 &        10395 && 19 &          4818{\,}35250\\
3 &    1 && 13 &   1{\,}35564 && 21 &         83083{\,}61040\\
5 &    9 && 15 &  19{\,}27926 && 23 &   15{\,}03096{\,}79212\\
7 &   84 && 17 & 295{\,}24716 && 25 &  283{\,}65681{\,}18720\\
9 &  882 &&\\
\bottomrule
\end{tabular}
\end{center}
\bigskip

 Taking $\nu^{\lim}(\tau)$ of the form $P(\tau)  \left(\frac{1}{\sqrt{2 \pi}} e^{-\tau^2/2}\right)$ with $P$ an odd polynomial of degree~$5$, we want 
 \[
 \int_\R  \tau^{2n+1} \cdot \nu^{\lim}(\tau) \,d\tau = -2 \bb_{2 n+1}(\M_3),
 \] 
for $n=0,1$ and $2$, and one finds that
\[
\nu^{\lim}(\tau)= \left(1/60\, \tau^5-1/2\, \tau^3+5/4\, \tau\right) \left(\frac{1}{\sqrt{2 \pi}} e^{-\tau^2/2}\right).
\]
Remarkably, the moments of $\nu^{\lim}(\tau)$ still agree with $-2 \bb_{2n+1}(\M_3)$ for $n=3,4$ and $5$. However, for $n=6$ we find that
$\int_\R \tau^{13} \cdot  \nu^{\lim}(\tau)\, d\tau = -2 \cdot 135135 \ne -2 \cdot \bb_{13}(\M_3)$. 

\input multigraph-comp.tex

In Figure~\ref{fig:comp} we see a comparison between the graph of points $\{p_{t,53}\}_{0 \leq t \leq 42}$ and the functions $\calV^{\lim}(\tau)$ and $\nu^{\lim}(\tau)$, in favor of the latter.

\nocite{Leontev2006russian}
\bibliographystyle{alphaurl}
\bibliography{shorterbib}

\end{document}